\documentclass[11pt]{amsart}
\usepackage{epsfig}

\topskip  \usepackage{epsfig}
\usepackage{amssymb}

\newtheorem{theorem}{Theorem}[section]
\newtheorem{prop}[theorem]{Proposition}
\newtheorem{lemma}[theorem]{Lemma}
\newtheorem{defn}[theorem]{Definition}
\newtheorem{corollary}[theorem]{Corollary}
\newcommand\beq{\begin{equation}}
\newcommand\eeq{\end{equation}}
\newcommand\bce{\begin{center}}
\newcommand\ece{\end{center}}
\newcommand\bea{\begin{eqnarray}}
\newcommand\eea{\end{eqnarray}}
\newcommand\ben{\begin{enumerate}}
\newcommand\een{\end{enumerate}}
\newcommand\bit{\begin{itemize}}
\newcommand\eit{\end{itemize}}
\newcommand\brr{\begin{array}}
\newcommand\err{\end{array}}
\newcommand\bt{\begin{tabular}}
\newcommand\et{\end{tabular}}

\newcommand\nn{\nonumber}

\newcommand\ms{\medskip}

\renewcommand\S{{\mathcal S}}
\newcommand\X{{\mathcal X}}
\newcommand\A{{\mathcal A}}
\newcommand\I{{\mathcal I}}
\newcommand\B{{\mathcal B}}
\newcommand\W{{\mathcal W}}
\renewcommand\P{{\mathcal P}}
\newcommand\bwx{f}
\newcommand\baw{\varphi}
\newcommand\inv{\operatorname{inv}}
\newcommand\fp{\operatorname{fp}}
\newcommand\exc{\operatorname{exc}}
\newcommand\cyc{\operatorname{cyc}}

\newcommand\open{\begin{picture}(13, 10)(-2,0)
 \put(0,0){\line(1,1){10}}
 \put(5,5){\line(1,0){5}}
 \put(5,5){\line(0,1){5}}
\end{picture}}
\newcommand\close{\begin{picture}(13, 10)(-2,0)
 \put(0,0){\line(1,1){10}}
 \put(0,5){\line(1,0){5}}
 \put(5,0){\line(0,1){5}}
\end{picture}}
\newcommand\ubounce{\begin{picture}(13, 10)(-2,0)
 \put(0,0){\line(1,1){10}}
 \put(0,5){\line(1,0){5}}
 \put(5,5){\line(0,1){5}}
\end{picture}}
\newcommand\lbounce{\begin{picture}(13, 10)(-2,0)
 \put(0,0){\line(1,1){10}}
 \put(5,5){\line(1,0){5}}
 \put(5,0){\line(0,1){5}}
\end{picture}}
\newcommand\fixed{\begin{picture}(13, 10)(-2,0)
 \put(0,0){\line(1,1){10}}
 \put(5,5){\circle*{4}}
\end{picture}}

\begin{document}
\title{The $X$-class and almost-increasing permutations}
\maketitle

\begin{center}
{Sergi Elizalde \\
Department of Mathematics \\
Dartmouth College, Hanover, NH 03755\\[4pt]
sergi.elizalde@dartmouth.edu}

\end{center}

\begin{abstract}

In this paper we give a bijection between the class of permutations that can be drawn on an $X$-shape
and a certain set of permutations that appears in Knuth~\cite{Knu} in connection to sorting algorithms.
A natural generalization of this set leads us to the definition of {\em almost-increasing permutations},
which is a one-parameter family of permutations that can be characterized in
terms of forbidden patterns. We find generating functions for almost-increasing permutations by using their
cycle structure to map them to colored Motzkin paths. We also give refined enumerations with respect to the
number of cycles, fixed points, excedances, and inversions.\ms

\end{abstract}

\section{Introduction and background}\label{sec:intro}

Permutations can sometimes describe the geometry of certain
configurations of points in the plane. {\em Picture classes}, introduced by Waton~\cite{Wat}, are sets of permutations
that can be drawn on given shapes in the plane. These classes are closed under pattern containment, and sometimes can be
characterized as the set of permutations that avoid a finite list of patterns. For example, it is shown in~\cite{Wat} that
a permutation can be drawn on a circle if and only if it avoids a fixed list
of 16 patterns. Another class that is studied in~\cite{Wat} is the $X$-class, which is the set of permutations
that can be drawn on a pair of crossing lines forming an $X$-shape (see Section~\ref{sec:xclass} for a precise definition).
Waton characterized them in terms of forbidden patterns, and showed that their counting sequence has a rational generating function.

Interestingly, the same generating function also appears in Knuth~\cite{Knu}, where it enumerates permutations $\pi$ such that for every $i$ there
is at most one $j\le i$ with $\pi(j)>i$. These permutations have a characterization in terms of forbidden patterns as well. The problem of finding a bijection
between Knuth's permutations and the $X$-class was proposed in the open problem session at the conference on Permutation Patterns held in 2006 in Reykjavik.
The present paper originates from this open problem. The sought bijection between these two sets of permutations is given in Section~\ref{sec:bija1}, using an
encoding of permutations as words.

Knuth's permutations can be regarded intuitively as permutations whose entries are not too far from those of the identity
permutation. The definition admits a natural generalization that depends on a parameter $k$, where Knuth's permutations are simply the case $k=1$. We call the
permutations in this family
{\it almost-increasing permutations}. Their precise definition is given in Section~\ref{sec:aip}, where we also characterize them in terms of pattern
avoidance.

The enumeration of almost-increasing permutations, for any value of the parameter $k$, is done in Section~\ref{sec:enumaip}.
Our method consists in using what we call the {\em cycle diagrams} of permutations to map them to certain colored Motzkin paths.
In Section~\ref{sec:stat} we obtain multivariate generating functions that keep track of the number of cycles, fixed points, excedances, and inversions.
When the parameter $k$ goes to infinity, our generating functions for almost-increasing permutations become continued fractions enumerating all permutations
by these statistics. We also restrict our results to involutions.


\ms

Here is some notation about pattern avoidance that will be used in the paper. We denote by $\S_n$ be the set of permutations of
$\{1,2,\ldots,n\}$. Using the standard notion of pattern avoidance, we say that $\pi\in\S_n$ avoids $\sigma\in\S_m$
if there is no subsequence of $\pi$ whose entries are in the same relative order as the entries of $\sigma$.
$\S_n(\sigma)$ denotes the
set of permutations in $\S_n$ that avoid the pattern $\sigma$. More
generally, given $\sigma_i\in\S_{m_i}$ for $1\le i\le k$, $\S_n(\sigma_1,\sigma_2,\dots,\sigma_k)$ is the set of
permutations that avoid all the patterns
$\sigma_1,\sigma_2,\dots,\sigma_k$.

We will use the terms {\em fixed point}, {\em excedance}, and {\em
deficiency} to denote an entry of a permutation $\pi$ such that
$\pi(i)=i$, $\pi(i)>i$, or $\pi(i)<i$, respectively. Recall also that $(i,j)$ is an {\em inversion} of $\pi$ if $i<j$ and $\pi(i)>\pi(j)$.
The number of fixed points, excedances, inversions, and cycles of $\pi$ will be denoted by $\fp(\pi)$, $\exc(\pi)$, $\inv(\pi)$, and $\cyc(\pi)$, respectively.

\section{Almost-increasing permutations}\label{sec:aip}

\begin{defn}
For any fixed $k\ge0$, the set of {\em $k$-almost-increasing permutations} is defined as
$$\A^{(k)}_n= \{\pi\in\S_n\, :\, |\{j : j\le i \mbox{ and } \pi(j)>i\}|\le k \mbox{ for every } i \}.$$
\end{defn}

Note that $\A^{(0)}_n$ contains only the identity permutation
$12\dots n$, whereas $\A^{(\lfloor n/2\rfloor)}_n=\S_n$. Intuitively, the smaller $k$ is, the closer the
permutations in $\A^{(k)}_n$ are to the increasing permutation. This is the reason for the name almost-increasing.
It is useful to observe that, for any permutation
$\pi$, \beq\label{eq:sym} |\{j : j\le i \mbox{ and }
\pi(j)>i\}|=|\{j : j> i \mbox{ and } \pi(j)\le i\}|. \eeq This identity, which is illustrated in Figure~\ref{fig:almostincr},
follows by noticing that both sides of the equality are equal to
$i-|\{j : j\le i \mbox{ and } \pi(j)\le~i\}|$.

\begin{figure}[hbt]
\epsfig{file=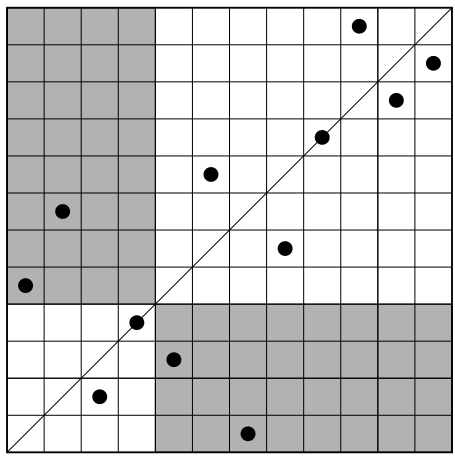,height=6cm}
\caption{\label{fig:almostincr} The permutation
$\pi=(5,7,2,4,3,8,1,6,9,12,10,11)\in\A^{(2)}_{12}$.}
\end{figure}

The set $\A^{(1)}_n=\{\pi\in\S_n : \mbox{for every } i \mbox{ there
is at most one } j\le i \mbox{ with } \pi(j)>i\}$ appears in Knuth's
book~\cite[Section 5.4.8, Exercise 8]{Knu} in connection to sorting
algorithms. It has an interesting characterization in terms of pattern avoidance as
$$\A^{(1)}_n=\S_n(3412,3421,4312,4321).$$
This result can be generalized to arbitrary $k$ as follows.

\begin{lemma}
Let $k\ge0$, and let
$\Sigma^{(k)}=\{\sigma\in\S_{2k+2}:\sigma_i>\sigma_j \mbox{ for
every } i\le k+1<j\}$. Then,
$$\A^{(k)}_n=\S_n(\Sigma^{(k)}).$$
\end{lemma}

\begin{proof}
If $\pi\in\S_n\setminus\A^{(k)}_n$, then there exist indices
$i,j_1,j_2,\ldots,j_{k+1}$ such that $j_1<j_2<\ldots<j_{k+1}\le i$
and $\pi(j_1),\pi(j_2),\ldots,\pi(j_{k+1})>i$. By
equation~(\ref{eq:sym}), this also implies that there are indices
$\ell_1,\ell_2,\ldots,\ell_{k+1}$ such that $i<\ell_1<\ell_2<\ldots<\ell_{k+1}$ and
$\pi(\ell_1),\pi(\ell_2),\ldots,\pi(\ell_{k+1})\le i$.
But now $\pi(j_u)>\pi(\ell_v)$ for all $1\le u,v\le k+1$, so $\pi(j_1)\dots\pi(j_{k+1})\pi(\ell_1)\dots\pi(\ell_{k+1})$ forms an
occurrence in $\pi$ of a pattern in $\Sigma^{(k)}$.

To prove the reverse inclusion, assume that $\pi\in\S_n$ has an
occurrence $\pi(i_1)\pi(i_2)\dots\pi(i_{2k+2})$ (with $i_1<i_2<\dots<i_{2k+2}$) of one of the
patterns in $\Sigma^{(k)}$. If
$\pi(i_1),\pi(i_2),\ldots,\pi(i_{k+1})>i_{k+1}$, then
$\pi\not\in\A^{(k)}_n$ and we are done. Otherwise,
$\pi(i_{k+2}),\pi(i_{k+3}),\ldots,\pi(i_{2k+2})\le i_{k+1}$, because
of the shape of the patterns in $\Sigma^{(k)}$. But then, again by
(\ref{eq:sym}),
there must be $k+1$ indices $j_1,\ldots,j_{k+1}\le i_{k+1}$ such
that $\pi(j_1),\ldots,\pi(j_{k+1})>i_{k+1}$. Thus,
$\pi\not\in\A^{(k)}_n$ in this case either.
\end{proof}

\section{The $X$-class}\label{sec:xclass}

Consider two crossing lines in the plane with slopes $1$ and $-1$,
forming an $X$-shape. Place $n$ points anywhere on these lines, with no two of them having the same $x$- or $y$-coordinate,
and label them $1,2,\ldots,n$ by increasing $y$-coordinate. Reading the labels of
the points by increasing $x$-coordinate determines a permutation. The set of
permutations obtained in this way is called the $X$-class. This set
was studied by Waton~\cite{Wat}, who also showed that these are
precisely those permutations that avoid the patterns
$2143$, $2413$, $3142$, and $3412$. In other words, if we denote by $\X_n$
the set of permutations of length $n$ in the $X$-class, we have that $\X_n=\S_n(2143,2413,3142,3412)$.

\subsection{Enumeration of the $X$-class}\label{sec:enumxclass}

We will show that $\X_n$ is in bijection with the set
$\W_n$ of words of length $n-1$ over the alphabet
$\{W,E,L,R\}$ not containing any occurrence (in consecutive
positions) of $LE$ or $RW$, and always ending with a $W$ or an $E$.

It will be convenient to represent a permutation $\pi\in\S_n$ as an
$n\times n$ array that contains dots in positions $(i,\pi(i))$, for
$i=1,2,\dots,n$ (see Figures~\ref{fig:almostincr}
and~\ref{fig:xclass}). We use the convention that column numbers
increase from left to right and row numbers increase from bottom to
top. The following property of the permutations in the $X$-class
will be useful to define the bijection.

\begin{figure}[hbt]
\epsfig{file=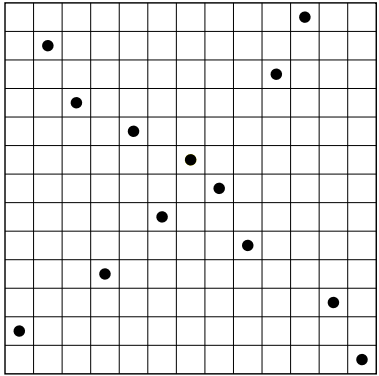,height=5cm} \caption{\label{fig:xclass} A
permutation in the $X$-class.}
\end{figure}

\begin{lemma}\label{lem:corner}
If $\pi\in\X_n$, then the array of $\pi$ contains a dot in at least
one of the four corners.
\end{lemma}

\begin{proof}
The statement of the lemma is equivalent to the fact that at
least one of these four conditions holds: $\pi(1)=1$, $\pi(1)=n$,
$\pi(n)=1$, $\pi(n)=n$. Let us assume for contradiction that none of
them holds. Let $i$ be the index such that $\pi(i)=1$, and let $j$
be such that $\pi(j)=n$. We have that $1<i,j<n$. If $i<j$, the
subsequence $\pi(1)\pi(i)\pi(j)\pi(n)$ forms an occurrence of $2143$
or $3142$, depending on whether $\pi(1)<\pi(n)$ or $\pi(1)>\pi(n)$,
respectively. Similarly, if $j<i$, then $\pi(1)\pi(j)\pi(i)\pi(n)$
forms an occurrence of $2413$ or $3412$. In any case, this
contradicts the fact that $\pi\in\X_n$.
\end{proof}

If the row and column containing a corner dot in the array of
$\pi\in\X_n$ are removed, then the resulting array determines a
permutation in $\X_{n-1}$. Furthermore, if we take the array of any permutation
in $\X_{n-1}$ and we add a row and a column intersecting at a
corner, with a dot in that corner, then we obtain the array of a
permutation in $\X_n$. The reason is that the new corner dot cannot
be part of an occurrence of any of the patterns
$2143,2413,3142,3412$.

\ms

Next we describe the bijection $\bwx:\W_n\longrightarrow\X_n$. Let
a word $w\in\W_n$ be given. Starting from an empty $n\times n$ array,
read $w$ from left to right, and for each letter
place a dot in an array according to the following rule. Every time an
$L$ (resp. $R$, $W$, $E$) is read, place a dot in the lower-left
(resp. lower-right, upper-left, upper-right) corner of the unshaded
region, and shade the row and column of the new dot. After reading
the whole word $w$, only one square remains unshaded. Place a dot in
the unshaded square. We define $\bwx(w)$ to be the permutation whose
array is constructed in this way. See the right side of Figure~\ref{fig:awx} for an example.

The inverse of $\bwx$ can be easily described. We start with the
array of a permutation $\pi\in\X_n$, with no shaded squares. We will
successively shade some squares as we write a word $w\in\W_n$. At
each step, the permutation given by the dots in the unshaded area
belongs to the $X$-class, so by Lemma~\ref{lem:corner} one of its
corners must contain a dot. If two of the corners (which are
necessarily opposite) contain a dot, we choose the dot in the upper
corner. If this dot is in the lower-left (resp. lower-right,
upper-left, upper-right) corner of the unshaded region, we append an
$L$ (resp. $R$, $W$, $E$) to the word and we shade the row and column containing
the dot. We repeat this process until the unshaded area contains
only one square.

It is clear that the procedures described in the above two
paragraphs are the inverse of each other, so $\bwx$ is a bijection
between $\W_n$ and $\X_n$. Enumerating $\W_n$ is straightforward, as the following proposition shows.

\begin{prop}\label{prop:gfw}
The generating function for the sequence $b_n=|\W_n|=|\X_n|$ is
\beq 1+\sum_{n\ge1}b_n x^n=\frac{1-3x}{1-4x+2x^2}.\label{eq:gfx}\eeq
\end{prop}

\begin{proof}
We show that the sequence satisfies the recurrence
\beq\label{eq:recb} b_n=4b_{n-1}-2b_{n-2} \eeq for $n\ge3$. To see
this, notice that to construct a word in $\W_n$ we have 4 choices
for the first letter, which can be followed by any word in
$\W_{n-1}$, except for the case where the first two letters create
an occurrence of either $LE$ or $RW$, followed by a word in
$\W_{n-2}$. Using that the initial terms are $b_1=1$ and $b_2=1$, we
get the generating function above.
\end{proof}

\subsection{A bijection with lattice paths}

The $X$-class is also equinumerous to the following family of paths on the integer lattice
$\mathbb{Z}^2$. Let $\P_n$ be the set of paths from $(0,0)$ to $(2n-2,0)$
with steps $U=(1,1)$ and $D=(1,-1)$ whose second coordinate
satisfies $|y|\le3$. We next describe a simple bijection between
$\W_n$ and $\P_n$.

Given a word $w\in\W_n$, we divide it into blocks by inserting a divider after every $E$ and after
every $W$. From each block we construct a piece of the path returning to the $x$-axis according to
the following rules.
\begin{itemize}
\item If the block contains only one letter, the corresponding piece of path is $UD$ if the
letter is an $E$, and $DU$ if the letter is a $W$.
\item Otherwise, the block ends in $RE$ or $LW$, preceded by a sequence of $R$s and $L$s.
\bit \item If the block ends in $RE$, start the corresponding piece
of path with $UU$, followed by $UD$ (resp. $DU$) for each $R$ (resp.
$L$) in the sequence, from left to right, not including the $R$
immediately before the $E$, and end with $DD$. \item If the block
ends in $LW$, start the piece of path with $DD$, followed by $DU$
(resp. $UD$) for each $R$ (resp. $L$) in the sequence, from left to
right, not including the $L$ immediately before the $W$, and end
with $UU$. \eit
\end{itemize}

For example, if $w=WRLWERLRE$, the path is $DUDDDUUUUDUUUDDUDD$,
which is drawn in Figure~\ref{fig:path3}. This bijection has additional properties.
Recall that {\em return} of a path is a step whose right
endpoint is on the $x$-axis (for example, the path in
Figure~\ref{fig:path3} has four returns). In our
construction, $E$s in the word are mapped to returns
from above in the path, and $W$s are mapped to returns from below.
Also, $R$s not followed by an $E$ correspond to points in the path
with $|y|=3$.

\begin{figure}[hbt]
\epsfig{file=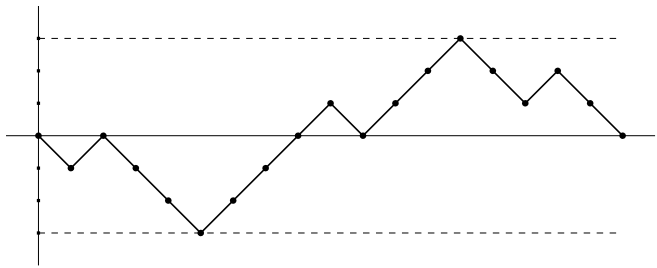,height=35mm} \caption{\label{fig:path3} An
element of $\P_{10}$.}
\end{figure}

It is an easy exercise to check directly that the generating function for these paths is
given by equation (\ref{eq:gfx}) as well.

\section{From the $X$-class to almost-increasing permutations}\label{sec:bija1}

In this section we show that there is a bijection between the $X$-class and $1$-almost-increasing permutations.
In terms of pattern avoidance, we will prove that $|\S_n(2143,2413,3142,3412)|=|\S_n(3412,3421,4312,4321)|$ for all $n$.
This equality follows from \cite{Knu,Wat}, but no bijective proof of it was known.

In Section~\ref{sec:enumxclass} we described a bijection $\bwx:\W_n\longrightarrow\X_n$. Next we present a bijection $\baw$ between $\A^{(1)}_n$ and $\W_n$ which, composed with
$\bwx$, will give the desired bijection between $\A^{(1)}_n$ and $\X_n$.

To construct $\baw(\pi)$ for a given $\pi\in\A^{(1)}_n$, we read its
entries $\pi(1),\pi(2),\ldots,\pi(n-1)$ from left to right, and
write a sequence of letters in the alphabet $\{W,E,L,R\}$ as
follows.

Assume that $\pi(i)$ is the current entry. Let $m=|\{j>i :
\pi(j)<\pi(i)\}|$. We consider three cases:
\begin{itemize}
\item If $m=0$, write a $W$.
\item If $m=1$, write an $E$.
\item If $m\ge2$, read the next $m-1$ entries, namely
$\pi(j)$ for $j=i+1,\ldots,i+m-1$. For each entry $\pi(j)$, write an
$R$ if $\pi(j)=j$, and write an $L$ otherwise. Next, if the last
letter (the one corresponding to $\pi(i+m-1)$) was an $R$ (resp. an
$L$), write an $E$ (resp. a $W$) right after it.
\end{itemize}
The entry to the right of the last one that was read becomes the new current entry. We repeat this process until entry
$\pi(n-1)$ has been read.

\ms

For example, let $\pi=(5,2,1,4,3,7,6,10,8,13,11,9,12)\in\A^{(1)}_{13}$. We
start with $\pi(1)=5$, so $m=4$. The next $m-1$ entries are
$\pi(2)=2$, $\pi(3)=1$, and $\pi(4)=4$, so the first letters of
$\baw(\pi)$ are $RLR$, followed by an $E$. The next entry is $3$,
and now $m=0$, so we write a $W$. Next we look at the entry $7$, so
$m=1$ and we write an $E$. For the entry $6$, $m=0$ again so we
write a $W$. The next entry is $10$, which means that $m=2$, so we
have to look at $\pi(9)=8$ and write an $L$, followed by a $W$ (see Figure~\ref{fig:awx}).
We now read $13$, so $m=3$, and we have to read the entries
$\pi(11)=11$, $\pi(12)=9$, and write $RL$, followed by a $W$.
At this point, all the entries up to $\pi(n-1)$ have been read, so
the word obtained is $\baw(\pi)=RLREWEWLWRLW$, and
$\bwx(\baw(\pi))=(2, 12, 10, 4, 9, 6, 8, 7, 5, 11, 13, 3, 1)$. As
another example, if $\pi=(2,6,1,4,5,8,7,3,10,9,11)$, then
$\baw(\pi)=ELRREREWEW$, and $\bwx(\baw(\pi))=(1, 8, 6, 5, 7, 9, 4,
10, 3, 2, 11)$.

\begin{figure}[hbt]
\epsfig{file=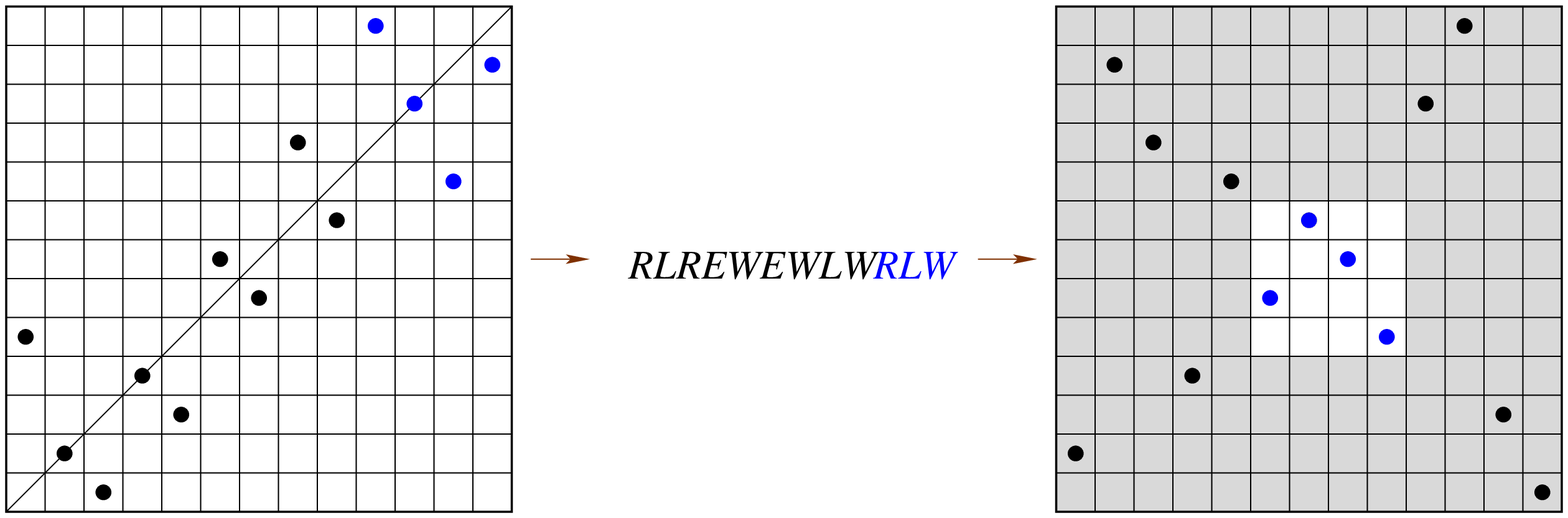,height=5cm} \caption{\label{fig:awx} The
construction of $\baw(\pi)$ and $\bwx(\baw(\pi))$.}
\end{figure}

\begin{prop}
The map $\baw$ described above is a bijection between $\A^{(1)}_n$ and $\W_n$.
\end{prop}

\begin{proof}
Consider the following recursive construction of permutations in $\A^{(1)}_n$. Let $\pi\in\A^{(1)}_n$,
represented as an $n\times n$ array of dots as described above.

If $\pi(1)=1$ (which corresponds to $m=0$ in the bijection),
deleting the first row and the first column of the array we obtain a
permutation $\pi'\in\A^{(1)}_{n-1}$, and every permutation in $\A^{(1)}_{n-1}$
can be obtained in this way. In this case, $\baw$ maps $\pi$ to the
word $W\baw(\pi')$.

If $\pi(1)=2$ (which corresponds to $m=1$ in the bijection), again
deleting row 1 and column 2 in the array,
we obtain the array of a permutation $\pi'\in\A^{(1)}_{n-1}$, and every
permutation in $\A^{(1)}_{n-1}$ can be obtained in this way. In this case,
$\baw$ maps $\pi$ to the word $E\baw(\pi')$.

If $\pi(1)\ge3$, we let $m=\pi(1)-1$. By definition of $\A^{(1)}_n$,
this forces $\pi(2)\le2, \pi(3)\le3, \ldots, \pi(m)\le m$. Besides,
deleting the rows and columns of the dots in the array corresponding
to the first $m$ entries of $\pi$, we obtain the array of a
permutation $\pi'\in\A^{(1)}_{n-m}$, and every permutation in
$\A^{(1)}_{n-m}$ can be obtained in this way. Note that there are
two choices for $\pi(2)$, namely $\{1,2\}$, two choices for
$\pi(3)$, namely $\{1,2,3\}\setminus\{\pi(2)\}$, and in general two
choices for each $\pi(j)$ for $k=2,\ldots,m$, one being $m$ and the
other being the value in $\{1,\ldots,m-1\}$ that is not attained by
any of $\pi(2),\ldots,\pi(m-1)$. These choices determine whether
each one of the first $m-1$ entries of $\baw(\pi)$ is an $R$ or an
$L$. The $m$th entry, an $E$ or a $W$, is then determined by the
condition that the word must avoid occurrences of $LE$ or $RW$. The
remaining $n-m$ letters of $\baw(\pi)$ are just $\baw(\pi')$.
\end{proof}

The above recursive description of $\A^{(1)}_n$ shows that
the numbers $a^{(1)}_n=|\A^{(1)}_n|$ satisfy the recurrence
$$a^{(1)}_n=2a^{(1)}_{n-1}+\sum_{m=2}^{n-1}2^{m-1}a^{(1)}_{n-m}$$ for $n\ge2$, with
initial condition $a_1=1$. This recurrence is equivalent to
(\ref{eq:recb}). The above proof implies that the sets $\W_n$ admit a
parallel recursive construction: any $w\in\W_n$ can be written as
either a $W$ or an $E$ followed by a word in $\W_{n-1}$, or a
sequence of $m-1$ (with $m\ge2$) $R$'s and $L$'s followed by the
letter in $\{W,E\}$ that does not create an occurrence of $LE$ or
$RW$, followed by a word in $\W_{n-m}$.

\section{Enumeration of almost-increasing permutations}\label{sec:enumaip}

For $k\ge0$, $n\ge1$, let $a^{(k)}_n=|\A^{(k)}_n|$, and let
$a^{(k)}_0=1$ by convention. For each $k$, define the generating
function
$$A^{(k)}(x)=\sum_{n\ge0}a^{(k)}_n x^n.$$
An expression for $A^{(1)}(x)$ follows from~\cite{Knu} or, alternatively, from Proposition~\ref{prop:gfw} and the
bijection in Section~\ref{sec:bija1}.

\begin{corollary}\label{cor:gfa1}
$$A^{(1)}(x)=\frac{1-3x}{1-4x+2x^2}.$$
\end{corollary}

In this section we generalize this result by finding simple expressions for $A^{(k)}(x)$ for any $k$.
As we will see, all these generating functions are rational. Similar expressions for $A^{(k)}(x)$
have been found by Atkinson~\cite{Atk} using inclusion-exclusion to obtain recurrence relations for the numbers $a^{(k)}_n$.

\subsection{A map to Motzkin paths}

Recall that a Motzkin path of length $n$ is a lattice path from $(0,0)$ to $(n,0)$ with steps $U=(1,1)$,
$D=(1,-1)$, $L=(1,0)$ with the condition that it never goes below the $x$-axis. The {\em height} of a step is
the $y$-coordinate of its right endpoint. The height of a path is the maximum height of any of its
steps.
A key ingredient in the enumeration of almost-increasing permutations will be a map from permutations to
Motzkin paths.

Given the representation of a permutation $\pi$ as an $n\times n$
array, we will depict its cycles in the following way. Take any index
$i_1\in\{1,2,\dots,n\}$. If $i_1$ is a fixed point, then it forms a
cycle of length 1. Otherwise, let $i_2=\pi(i_1)$, and draw a vetical
line in the array from the center of the square $(i_1,i_1)$ to the
center of the square $(i_1,i_2)$, followed by a horizontal line from
$(i_1,i_2)$ to $(i_2,i_2)$. Now we let $i_3=\pi(i_2)$ and repeat
the process, until we eventually return to $(i_1,i_1)$ after going
through all the elements of the cycle containing $i_1$. We do this
for each cycle of $\pi$, obtaining a picture like the example in
Figure~\ref{fig:cycles}. We call this the {\em cycle diagram} of
$\pi$.

\begin{figure}[hbt]
\epsfig{file=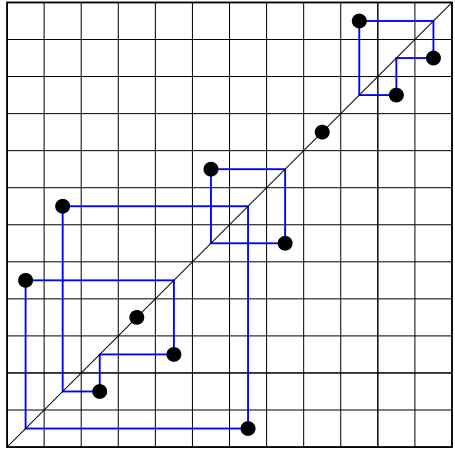,height=6cm}
\caption{\label{fig:cycles} The cycle diagram of $\pi=(5,7,2,4,3,8,1,6,9,12,10,11)$.}
\end{figure}

The squares in the diagonal of the cycle diagram, namely, those of the form $(i,i)$, can be
classified into five types: \bit
\item a {\em fixed point} \fixed,
\item an {\em opening bracket} \open,
\item a {\em closing bracket} \close,
\item an {\em upper bounce} \ubounce,
\item a {\em lower bounce} \lbounce.
\eit
 The sequence of types of the squares in the diagonal of the cycle diagram of $\pi$, read from
bottom-left to top-right, will be called the {\em diagonal sequence}
of $\pi$, and denoted $D(\pi)$. Note that
$D(\pi)\in\{$\fixed,\open,\close,\ubounce,\lbounce$\}^n$. We map
the permutation to a Motzkin path of length $n$ by considering its
diagonal sequence and drawing an up step $U$ for each \open, a
down step $D$ for each \close, and a level step $L$ for each \fixed, \ubounce\ or \lbounce.
Let us denote $\theta(\pi)$ the path
defined in this way. For example, if $\pi$ is the permutation in Figure~\ref{fig:cycles}, $\theta(\pi)$ is the path in Figure~\ref{fig:motzkin}.

\begin{figure}[hbt]
\epsfig{file=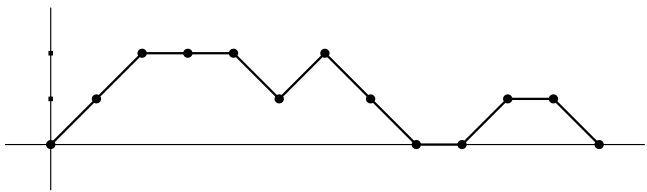,height=25mm}
\caption{\label{fig:motzkin} The Motzkin path $\theta(\pi)$, where $\pi=(5,7,2,4,3,8,1,6,9,12,10,11)$.}
\end{figure}

\begin{prop}\label{prop:height}
Let $\pi\in\S_n$. Then $\pi\in\A^{(k)}_n$ if and only if the height
of $\theta(\pi)$ is at most $k$.
\end{prop}

\begin{proof}
For $1\le i\le n$, let $h_i=|\{j : j\le i \mbox{ and } \pi(j)>i\}|$.
The key observation is that the height of the $i$-th step of
$\theta(\pi)$ is precisely $h_i$. This can be checked by induction on $i$, since
$h_{i}-h_{i-1}$ is $1$, $-1$, or $0$ depending on whether there is
an opening bracket, a closing bracket, or any other symbol in
position $(i,i)$, respectively.
\end{proof}

Despite the above useful property, one disadvantage of the map
$\theta$ is that it is not injective. To fix this problem, we will
modify the map using colored steps in the Motzkin paths.

\subsection{Bijections to colored Motzkin paths}

Let us first focus on the case $k=1$.

\begin{prop}
There is a bijection $\psi_1$ between $\A^{(1)}_n$ and colored
Motzkin paths of length $n$ and height at most 1 where the $L$ steps
at height 1 can receive three colors.
\end{prop}

\begin{proof}
For $\pi\in\A^{(1)}_n$, Proposition~\ref{prop:height} implies that
in $D(\pi)$, the positions of opening and closing brackets
alternate, starting with an opening bracket and ending with a
closing bracket. From the construction of the cycle diagram, one sees that in
the intervals between an opening and a closing bracket, any
combination of the other three types of squares (fixed point, upper
bounce, and lower bounce) is possible, while in the intervals not
enclosed by brackets, only fixed points can occur. Besides, any
sequence in \{\fixed,\open,\close,\ubounce,\lbounce\} satisfying
these conditions uniquely determines the permutation
$\pi\in\A^{(1)}_n$ that it came from.

Define $\psi_1$ to be a variation of the map $\theta$ where each
level step at height~1 of the Motzkin path receives once of three
different colors, depending on whether the corresponding element in
the diagonal sequence is a fixed point, an upper bounce, or a lower
bounce. Then $\psi_1$ is the desired bijection.
\end{proof}

Enumerating these colored Motzkin paths (see~\cite{Fla}) we recover the
expression for $A^{(1)}(x)$:
$$A^{(1)}(x)=\frac{1}{1-x-\dfrac{x^2}{1-3x}}=\frac{1-3x}{1-4x+2x^2}.$$

\ms

The above idea can be
generalized to enumerate $\A^{(k)}_n$ for any~$k$. By
Proposition~\ref{prop:height} we know that if $\pi\in\A^{(k)}_n$, then
the Motzkin path $\theta(\pi)$ has height at most $k$. A variation
of the map $\theta$ will produce a bijection between
$\A^{(k)}_n$ and certain colored Motzkin paths.

\begin{prop}\label{prop:bijpsik}
There is a bijection $\psi_k$ between $\A^{(k)}_n$ and colored
Motzkin paths of length $n$ and height at most $k$, where for each
$h$, \bit
\item each $L$ step at height $h$ receives one of $2h+1$ possible colors,
\item each $U$ step at height $h$ receives one of $h$ possible colors, and
\item each $D$ step at height $h-1$ receives one of $h$ possible colors.
\eit
\end{prop}

\begin{proof}
Given a permutation $\pi$, the underlying uncolored Motzkin path that
$\psi_k$ maps it to is just $\theta(\pi)$. It remains to show how the steps are colored.
For each entry in the diagonal sequence of $\pi$,
define its height to be the height of the corresponding step in
$\theta(\pi)$. Let $h_i$ be the height of the $i$-th entry. Let us
first consider entries that are not opening or closing brackets,
that is, those that correspond to $L$ steps. It is obvious that such
entries at height $0$ can only be fixed points. Such entries at
height one or more can be either fixed points, upper bounces, or
lower bounces. However, the map $\pi\mapsto D(\pi)$ is not injective
outside of $\A^{(1)}_n$.

For any given diagonal sequence $\Delta=D(\pi)$,
let us analyze the set $\{\pi'\in\S_n\,|\,D(\pi')=\Delta\}$.
Any permutation $\pi'$ with $D(\pi')=\Delta$ can be obtained in the following way. Place the entries of $\Delta$ on
the diagonal of an $n\times n$ array. Think of each symbol
\open,\close,\ubounce,\lbounce\ as a gadget with a vertical and a
horizontal ray that can be extended until they intersect another
ray. Read these symbols from bottom-left to top-right and proceed as
follows. \bit
\item Every time a \fixed\ is read, just place a dot (a fixed point) there.
\item Every time a \open\ is read, it creates a new open vertical ray in its column and a new open
horizontal ray in its row.
\item Every time a \ubounce\ is read, take any of the open vertical rays coming from the symbols read so far,
extend it upward until it intersects the leftward extension of the horizontal ray of the
\ubounce, and place a dot in the intersection. The intersected vertical
ray becomes closed after this, but the \ubounce\ creates a new open vertical ray in its column (see Figure~\ref{fig:reconstruct}).
\item Every time a \lbounce\ is read, proceed similarly with any of the horizontal rays that are open at that time.
\item Every time a \close\ is read, take any of the open vertical rays and any of the open horizontal rays,
extend them until they intersect the two extended
rays of the \close, and place a dot in each of the two
intersections. \eit The placed dots determine a permutation $\pi'$
with diagonal sequence $\Delta$.

\begin{figure}[hbt]
\epsfig{file=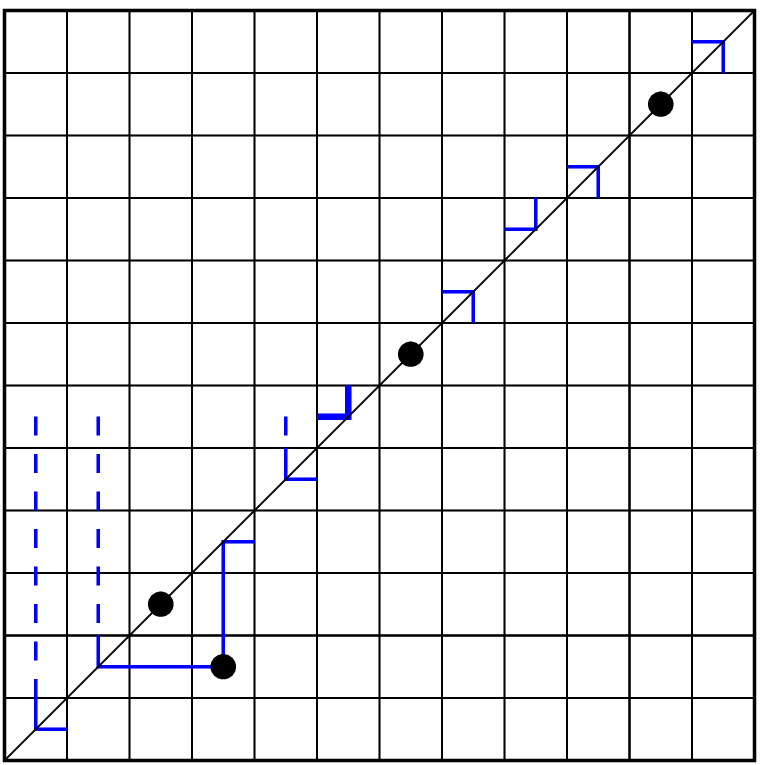,height=5cm}
\epsfig{file=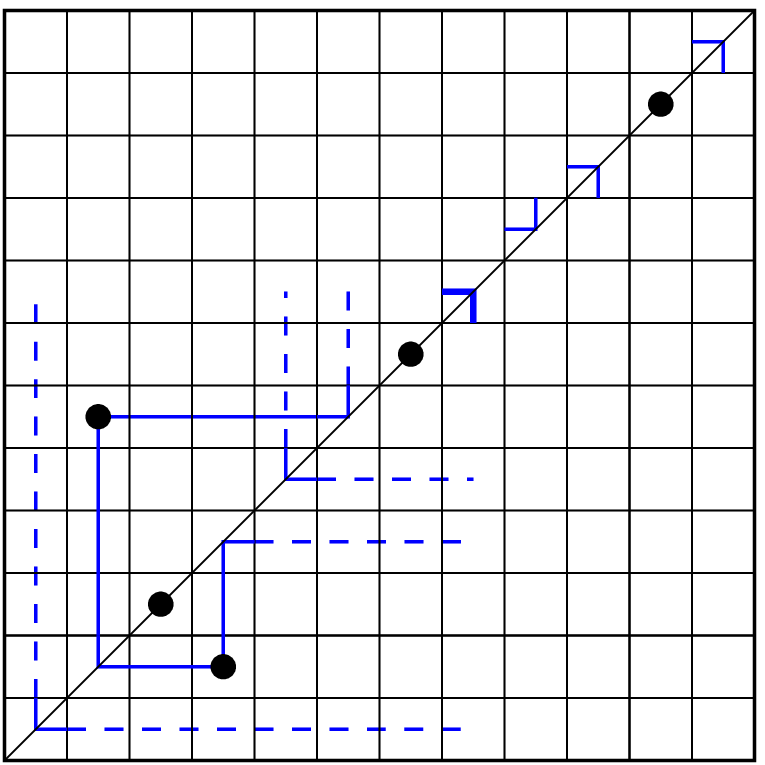,height=5cm}
\epsfig{file=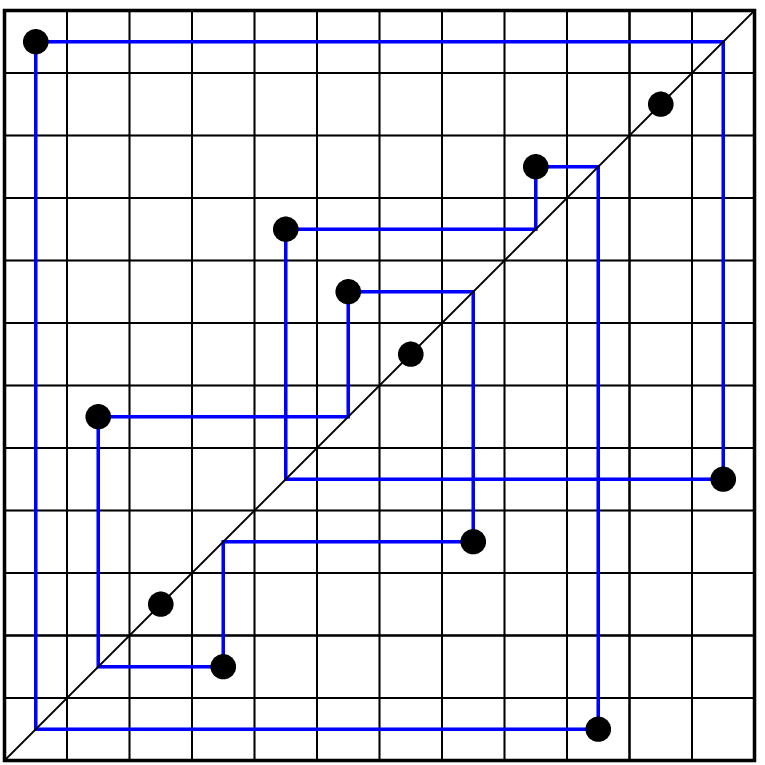,height=5cm}
\caption{\label{fig:reconstruct} Constructing a permutation with a
given diagonal sequence: the different possibilities for what ray to
extend when reading an upper bounce (left) and a closing bracket
(center).}
\end{figure}

To see how many permutations $\pi'$ satisfy $D(\pi')=\Delta$, note
that when a \ubounce\ or \lbounce\ is seen in position $(i,i)$, the
number of vertical (equivalently, horizontal) rays that are open at
that time equals $h_i$. Thus, the number of possibilities for which
ray to close is the height of the corresponding $L$ step in the
path. Similarly, when a \close is seen in position $(i,i)$, the
number of vertical (equivalently, horizontal) rays that are open at
that time is equal to $h_i+1$. So, the number of possibilities for
which rays to close is $(h_i+1)^2$, where $h_i$ is the height of the
corresponding $D$ step in the path.

This argument determines how many colors we need for each step of the Motzkin path
in order to obtain a bijection. For each $h\ge0$, level
steps at height $h$ of $\theta(\pi)$ can receive one of $2h+1$
colors, corresponding to the $h$ possibilities of which
rays to close when the symbol is a \ubounce, plus the $h$
possibilities when the symbol is a \lbounce, plus the case where the
symbol is a \fixed. For each $h\ge1$, down steps at height $h-1$ can receive
one of $h^2$ colors, corresponding to the
possibilities of which rays to close when a \close\ is read.
Instead of using $h^2$ colors for down steps at height $h-1$, another equivalent (and more symmetric) way to
obtain the $h^2$ factor is by coloring up steps at height $h$ with one of $h$ possible colors and down steps at
height $h-1$ with one of $h$ possible colors. This is the coloring that we use to define the bijection $\psi_k$.

The description of the inverse map $\psi_k^{-1}$ is clear. Given a colored Motzkin path, read its steps from left to right while
building the array of a permutation from the lower-left to the upper-right corner. For each step of the path, place a symbol in the diagonal of the array accordingly,
that is,
\bit \item if the step is a $U$, put a \open, creating an open vertical ray and an open horizontal ray;
\item if the step is a $D$, put a \close, and use the colors of that $D$ and its matching $U$ to determine which horizontal and which vertical ray to close, placing dots where the closed rays intersect the extended rays of the \close;
\item if the step is a $L$, use its color to determine whether to put a \fixed, a \ubounce, or a \lbounce, and in the last two cases, also to determine which rays to close and where to place a dot.
\eit
\end{proof}

This bijection reduces the enumeration $\A^{(k)}_n$ to finding a
generating function for colored Motzkin paths with bounded height.

\begin{theorem}\label{thm:gfak} Let $k\ge0$. The generating function for the numbers
$a^{(k)}_n$ is
$$A^{(k)}(x)=\frac{1}{1-x-\dfrac{x^2}{1-3x-\dfrac{4x^2}{1-5x-\dfrac{9x^2}{\bt{l}
$\ddots\hspace{27mm}\ddots$ \\
$\qquad1-(2k-1)x-\dfrac{k^2x^2}{1-(2k+1)x}$ \et}}}}.$$
\end{theorem}

\begin{proof}
By Proposition~\ref{prop:bijpsik}, $A^{(k)}(x)$ is the generating function for weighed Motzkin paths where $L$ steps at height $h$ have weight $2h+1$, and $U$ steps at height $h$ and $D$ steps at height $h-1$ have weight $h$.
Obtaining the generating function is now a straightforward application of the tools from~\cite{Fla}.
\end{proof}

For small values of $k$, the expressions of $A^{(k)}(x)$ as a
quotient of polynomials are \bea \nn
A^{(2)}(x)&=&\frac{1-8x+11x^2}{1-9x+18x^2-6x^3}, \\ \nn
A^{(3)}(x)&=&\frac{1-15x+58x^2-50x^3}{1-16x+72x^2-96x^3+24x^4}, \\
\nn
A^{(4)}(x)&=&\frac{1-24x+177x^2-444x^3+274x^4}{1-25x+200x^2-600x^3+600x^4-120x^5}.
\eea

These results agree with~\cite{Atk}, where Atkinson determines the coefficients of these polynomials from the recurrence relation satisfied by the $a^{(k)}_n$.

It is worth mentioning that $\psi_k$ can naturally be extended to a bijection between $\S_n$ and colored Motzkin paths of length $n$ with no height restriction, where the possible colors of the steps
at each height are given by the same rules as in Proposition~\ref{prop:bijpsik}.

\section{Statistics on almost-increasing permutations}\label{sec:stat}

Having a bijection $\psi_k$ between almost-increasing permutations and
colored Motzkin paths enables us to study the distribution of some statistics on
almost-increasing permutations. The following is a refinement of Theorem~\ref{thm:gfak} by considering
the number of cycles, the number of fixed points, and the number of excedances.


\begin{theorem}\label{thm:gfak_ref}
Let $k\ge0$, and let $$F^{(k)}(t,u,v,x)=\sum_{c,i,j,n\ge0}
|\{\pi\in\A^{(k)}_n :\, \cyc(\pi)=c,\, \fp(\pi)=i,\, \exc(\pi)=j\}|\ t^c u^i v^j x^n.$$ Then $F^{(k)}(t,u,v,x)=$
{\small
$$\frac{1}{1-tux-\dfrac{tvx^2}{1-((1+v)+tu)x-\dfrac{2(1+t)vx^2}{1-(2(1+v)+tu)x-\dfrac{3(2+t)vx^2}{\bt{l}
$\ddots\hspace{27mm}\ddots$ \\
$1-((k-1)(1+v)+tu)x-\dfrac{k(k-1+t)vx^2}{1-(k(1+v)+tu)x}$ \et}}}}.$$}
\end{theorem}

\begin{proof}
To calculate the contribution of each step of the colored Motzkin path to the number of cycles, fixed points, and excedances
of the permutation, we think of the array of the permutation as being built as the steps of the path are read from left to right, using the description of $\psi_k^{-1}$.
Every time a cycle, fixed point, or excedance is created, this will be reflected it in the generating function.
We begin by justifying that the contribution of a level step at height $h$ in the generating
function is $h(1+v)+tu$. Recall that level steps of the
Motzkin path correspond to \fixed, \ubounce, and \lbounce\ in the diagonal sequence of the
permutation. A level step at height $h$ can be receive $2h+1$ colors. One of these colors
indicates a fixed point in the permutation, which contributes $tu$ to the generating function.
Of the remaining $2h$ colors, half of them come from a \ubounce, which creates an excedance in
the permutation, while the other half come from a \lbounce, which produces a deficiency. This
explains the contribution $h(1+v)$.

\begin{figure}[hbt]
\epsfig{file=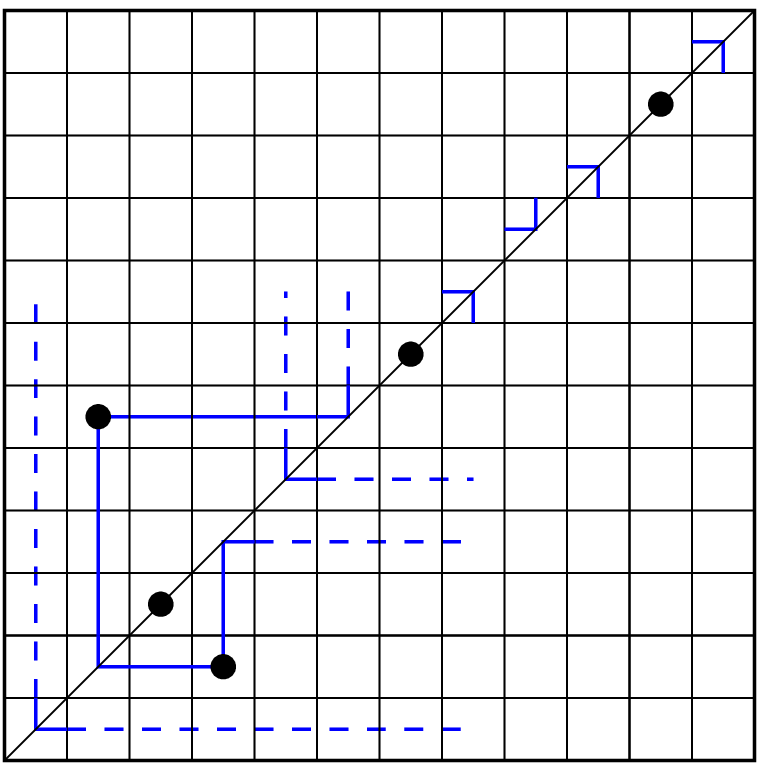,height=5cm}
\epsfig{file=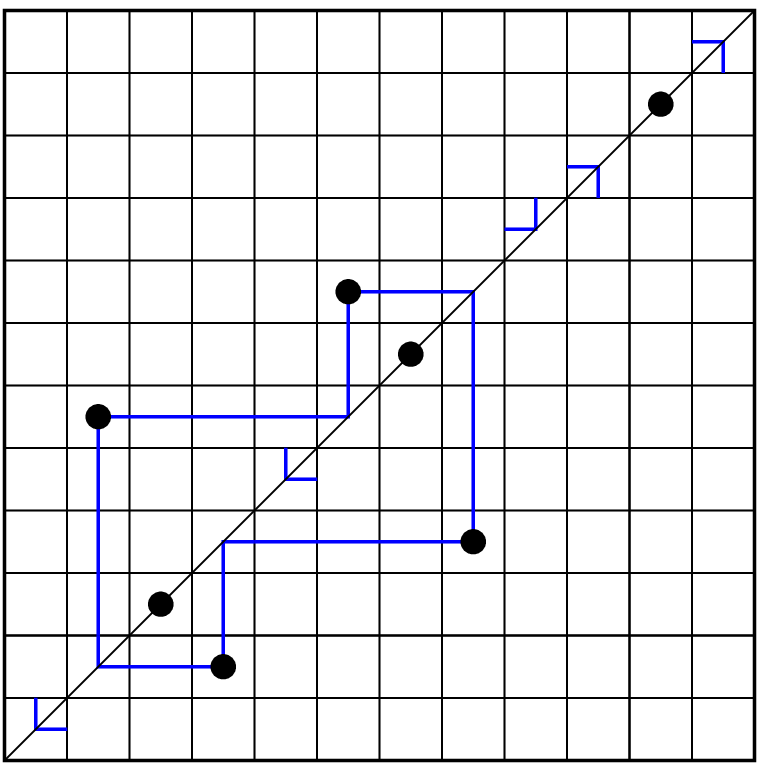,height=5cm}
\epsfig{file=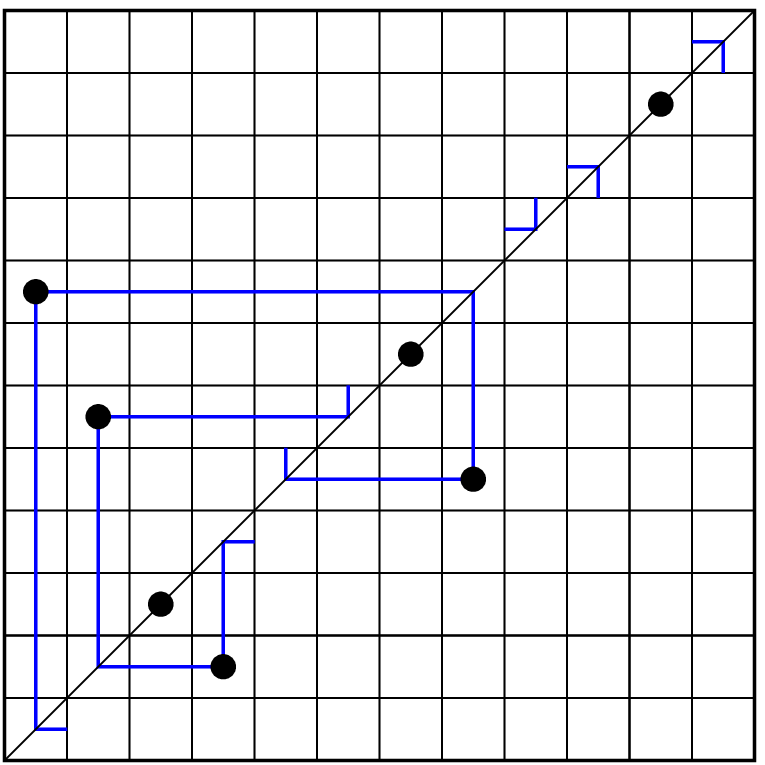,height=5cm}
\caption{\label{fig:difreconstruct} Two different ways of closing
the rays with a closing bracket: completing a cycle (center) or leaving it incomplete (right).}
\end{figure}

Next we show that the joint contribution of a $U$ at height $h$ and the
matching $D$ at height $h-1$ is $h(h-1+t)vx^2=(h^2-h+ht)v$. Recall
from the construction of $\psi_k$ that the $h^2$ possible ways of
coloring this pair of steps correspond to the choices among the $h$ vertical and the
$h$ horizontal rays that can be closed when the symbol \close
appears in the diagonal sequence. When building the array of the permutation from the colored Motzkin path,
these open rays can be thought of
as ``incomplete'' cycles. For each open vertical ray, there is a
unique open horizontal ray that belongs to the same cycle (see
Figure~\ref{fig:difreconstruct}). Closing these two rays
simultaneously completes the cycle. Closing two rays that belong to
different incomplete cycles merges them into one cycle;
this decreases by one the number of incomplete cycles, but does not
complete any cycle. Therefore, of the $h^2$ possible ways of choosing
the pair of rays to close, exactly $h$ of them complete a cycle. This justifies the
factor $h^2-h+ht$. The factor $v$ is explained by the fact that the dot placed at the intersection of the
closed vertical ray with the horizonal ray of the \close produces an excedance.
\end{proof}

Another statistic whose distribution we can obtain is the number of inversions. While we have not been able to use our method to
keep track of inversions and cycles simultaneously, the following result
gives the joint distribution of the number of inversions, the number of fixed points, and the number of excedances. We use the notation $[k]_q=1+q+q^2+\dots+q^{k-1}$.

\begin{theorem}\label{thm:gfak_inv}
Let $k\ge0$, and let $$G^{(k)}(q,u,v,x)=\sum_{r,i,j,n\ge0}
|\{\pi\in\A^{(k)}_n :\, \inv(\pi)=r,\, \fp(\pi)=i,\, \exc(\pi)=j\}|\ q^r u^i v^j x^n.$$ Then $G^{(k)}(q,u,v,x)=$
{\small $$\frac{1}{1-ux-\dfrac{vqx^2}{1((1+v)q+uq^2)x-\dfrac{vq^3(1+q)^2x^2}{\bt{l}$\ddots\hspace{27mm}\ddots$ \\
$1-((1+v)q^{k-1}[k-1]_q+uq^{2k-2})x-\dfrac{vq^{2k-1}[k]_q^2x^2}{1-((1+v)q^k[k]_q+uq^{2k})x}$ \et}}}.$$}
\end{theorem}

\begin{proof}
The idea of this proof is similar to that of Theorem~\ref{thm:gfak_ref}. We think of the
array of the permutation $\pi=\psi_k^{-1}(M)$ as being built as we go through the steps of the colored Motzkin path $M$. Each step determines what
symbol to place in the diagonal of the array, which rays to open and close, and where to place dots. We will consider that
a step of the path ``creates" an inversion in the permutation when the changes in the array produced by that step force the inversion to occur.

First we show that the contribution of a level step at height $h$ to the generating
function is $(1+v)q^h[h]_q+uq^{2h}$. If this step corresponds to a fixed point $i$ in the permutation, then $i$
will form an inversion with the $h$ dots that will be placed (further along the construction of~$\pi$) in the currently open vertical rays, above and to the left of $(i,i)$ in the array,
and also with the $h$ dots that will be placed in the currently open horizontal rays, below and to the right of $(i,i)$. This situation is depicted in Figure~\ref{fig:inv}(a).
The contribution in this case is $uq^{2h}$.

\begin{figure}[hbt]
\begin{tabular}{ccc}
\epsfig{file=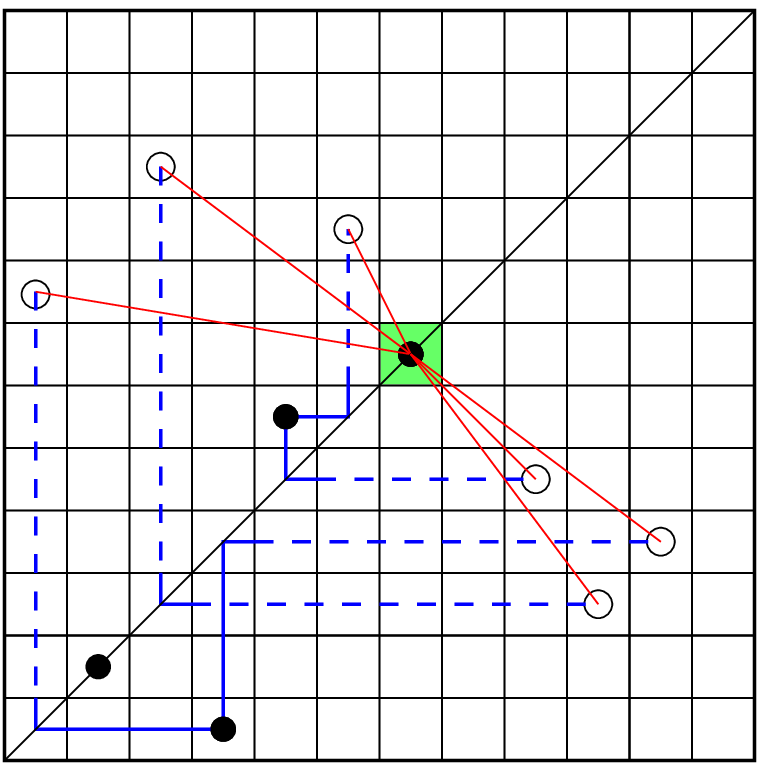,height=55mm} & \quad & \epsfig{file=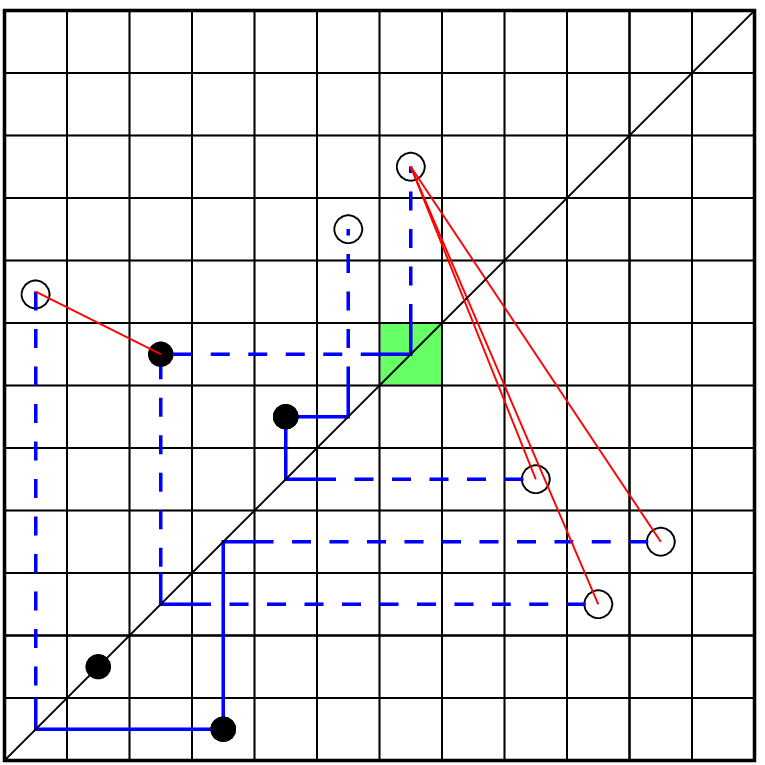,height=55mm} \\
(a) && (b) \vspace{3mm} \\
\epsfig{file=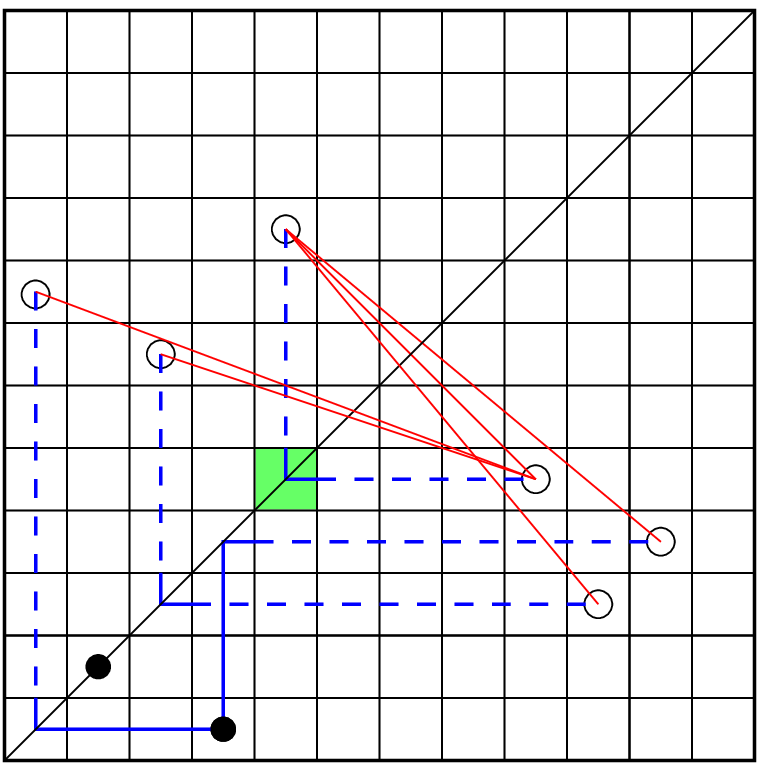,height=55mm} && \epsfig{file=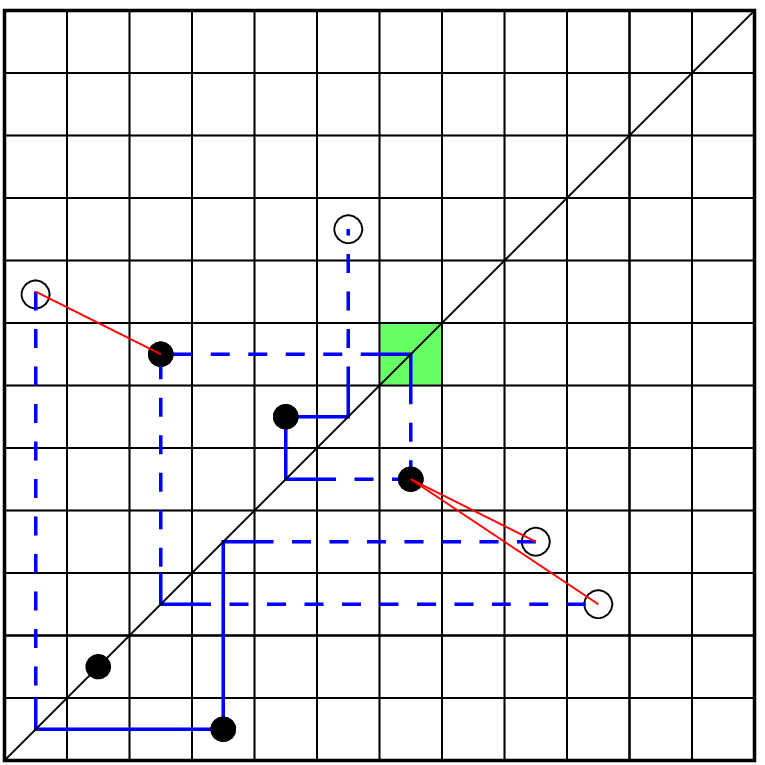,height=55mm}\\
(c) && (d)
\end{tabular}
\caption{\label{fig:inv} Inversions created by a fixed point (a), an upper bounce (b), an opening bracket (c), and a closing bracket (d), all corresponding to steps at height $h=3$ in the Motzkin path. The white dots represent entries that will be placed further along the construction of the permutation.}
\end{figure}

If the level step corresponds to a \ubounce\ in the array, then we have $h$ choices for which vertical ray to close and place a dot on.
If the $i$-th from the left of these vertical rays is chosen, the new placed dot will create inversions with the dots that will later be placed on the $i-1$ open vertical rays to the left of it, but not with the dots placed on the $h-i$ open vertical rays to its right. Additionally, the new open vertical ray produced by the
\ubounce\ will create (once a dot is placed on it) inversions with the future dots
on each one of the $h$ currently open horizonal rays. See Figure~\ref{fig:inv}(b) for an example with $i=2$. Since $1\le i\le h$, the contribution of a \ubounce\ to the generating function is $\sum_{i=1}^h vq^{(i-1)+h}=vq^h(1+q+\dots+q^{h-1})$, where the $v$ indicates that the new placed dot is an excedance.
The case of a \lbounce\ in the array is analogous, only that the new dot in this case does not produce an excedance, so the contribution is just $q^h(1+q+\dots+q^{h-1})$.

Next we show that the joint contribution of a $U$ at height $h$ and the
matching $D$ at height $h-1$ is $vq^{2h-1}[h]_q^2$. A $U$ in the path corresponds to a \open\ in the diagonal of the array of~$\pi$. The new vertical ray
emanating from this symbol forces
an inversion with each of the $h-1$ currently open horizontal rays (once dots are placed on them), and the new horizontal ray forces
an inversion with each of the $h-1$ currently open vertical rays. Additionally, the two future dots on the two new rays will also form an inversion pair. Figure~\ref{fig:inv}(c) represents this situation. The contribution of a \open\ is therefore $q^{(h-1)+(h-1)+1}=q^{2h-1}$.
Finally, a $D$ in the path corresponds to a \close\
in the array of $\pi$, which can close any one of the $h$ currently open vertical rays and any one of the $h$ currently open horizontal rays. If the $i$-th vertical ray
from the left is closed, the placed dot will create inversions with the future dots on the $i-1$ open vertical rays to its left. Similarly, if the $j$-th horizontal ray
from the bottom is closed, the placed dot will create inversions with the future dots on the $j-1$ open horizontal rays below it. Figure~\ref{fig:inv}(d) shows an
example with $i=2$ and $j=3$. Since $1\le i,j\le h$, the contribution of a \close\ to the generating function is $\sum_{i,j=1}^h vq^{(i-1)+(j-1)}=v(1+q+\dots+q^{h-1})^2$, where again the $v$ indicates that one of the two placed dots is an excedance.
\end{proof}

Finally, it is not difficult to restrict our results to involutions. Let $\I^{(k)}_n=\{\pi\in\A^{(k)}_n : \pi=\pi^{-1}\}$ be the set of $k$-almost-increasing involutions of length $n$.

\begin{corollary}\label{thm:gfak_invltn}
Let $k\ge0$,and let $$H^{(k)}(q,u,v,x)=\sum_{r,i,j,n\ge0}
|\{\pi\in\I^{(k)}_n :\, \inv(\pi)=r,\, \fp(\pi)=i,\, \exc(\pi)=j\}|\, q^r u^i v^j x^n.$$ Then $H^{(k)}(q,u,v,x)=$
{\small $$\frac{1}{1-ux-\dfrac{vqx^2}{1-uq^2x-\dfrac{vq^3(1+q^2)x^2}{\bt{l}$\ddots\hspace{27mm}\ddots$ \\
$1-uq^{2k-2}x-\dfrac{vq^{2k-1}(1+q^2+q^4+\dots+q^{2k-2})x^2}{1-uq^{2k}x}$ \et}}}.$$}
\end{corollary}

\begin{proof}
A permutation $\pi$ is an involution if and only if its array is symmetric with respect to the diagonal from the bottom-left to the top-right corner. This
implies that its diagonal sequence $D(\pi)$ does not contain any \ubounce\ or \lbounce\ symbols, and also that every \close in the array closes a vertical and a horizontal ray that are symmetric with respect to the diagonal.
This allows us to restrict $\psi_k$ to a bijection between $\I^{(k)}_n$
and colored Motzkin paths of length $n$ and height at most $k$, where $L$ and $U$ steps can only receive one color, and $D$ steps at height $h-1$ receive one of
$h$ possible colors, corresponding to the $h$ choices of which rays to close when a \close is placed in the diagonal, since the choice of horizontal ray determines
the choice of vertical ray.

To find the generating function for involutions with respect to the number of inversions, fixed points, and excedances, we argue as in the proof of Theorem~\ref{thm:gfak_inv}, with the following two changes. First, we have to exclude the contributions of colored $L$ steps corresponding to \ubounce\ or \lbounce\
symbols. This kills the $(1+v)q^h[h]_q$ terms, so an $L$ step at height $h$ only contributes $uq^{2h}$. Second, we have to take into account that when a \close\ closes the $i$-th open vertical ray from the left, it must also close the $i$-th open horizontal ray from the bottom, therefore creating $2(i-1)$ inversions. So the contribution of a \close\ is now $\sum_{i=1}^h vq^{2(i-1)}$.
\end{proof}

We conclude by mentioning that if we let $k$ go to infinity, Theorem~\ref{thm:gfak_ref} gives a continued fraction
expression for the generating function of all permutations with respect to the number of cycles, fixed points,
and excedances. By Foata's correspondence~\cite{Foa}, this also provides the enumeration of permutations by the
number of left-to-right minima and descents. Similarly, taking $k\rightarrow\infty$ in Theorem~\ref{thm:gfak_inv} (resp. Corollary~\ref{thm:gfak_invltn}),
we get a continued fraction that enumerates all permutations (resp. all involutions) by the inversion number and the number of fixed points and excedances.

\end{document}